\documentclass[leqno]{amsart}
\usepackage{amsmath}
\usepackage{amssymb}
\usepackage{amsthm}
\usepackage{enumerate}
\usepackage[mathscr]{eucal}
\usepackage{graphicx}
\theoremstyle{plain}
\newtheorem{theorem}{Theorem}[section]
\newtheorem{prop}[theorem]{Proposition}

\theoremstyle{definition}
\newtheorem{definition}{Definition}[section]
\newtheorem{remark}{Remark}[section]

\usepackage[pagewise]{lineno}
\begin{document}

\title[Smooth operators and B-P-B property]{Smooth points in operator spaces and some Bishop-Phelps-Bollob$\acute{A}$s type theorems in Banach spaces}
\author[Debmalya Sain]{Debmalya Sain}

\newcommand{\acr}{\newline\indent}

\address{\llap{\,}Department of Mathematics\acr
                              Indian Institute of Science\acr
															Bengaluru\acr
															Karnataka 560012\acr
                              INDIA}
\email{saindebmalya@gmail.com}

\thanks{The research of the author is sponsored by Dr. D. S. Kothari Postdoctoral Fellowship. The author feels elated to acknowledge the blissful and motivating presence of his junior brother Mr. Anubhab Ray, in every sphere of his life!} 


\begin{abstract}
We introduce the notion of approximate norm attainment set of a bounded linear operator between Banach spaces and use it to obtain a complete characterization of smooth points in the space of compact linear operators, provided the domain space is reflexive and Kadets-Klee. We also apply the concept to characterize strong BPB property (sBPBp) of a pair of Banach spaces. We further introduce uniform $ \epsilon- $BPB approximation of a bounded linear operator and uniform strong BPB property (uniform sBPBp) with respect to a given family of norm one linear operators and explore some of the relevant properties to illustrate its connection with earlier studies on Bishop-Phelps-Bollob$\acute{a}$s type theorems in Banach spaces. It is evident that our study has deep connections with the study of smooth points in operator spaces. We obtain a complete characterization of uniform sBPBp for a pair of Banach spaces, with respect to a given family of norm one bounded linear operators between them. As the final result of this paper, we prove that if $ \mathbb{X} $ is a reflexive Kadets-Klee Banach space and $ \mathbb{Y} $ is any Banach space, then the pair $ (\mathbb{X},\mathbb{Y}) $ has sBPBp for compact operators. Our results extend, complement and improve some of the earlier results in this context.     
\end{abstract}

\maketitle

\section{Introduction.}
Bishop-Phelps theorem \cite{B}, unarguably one of the cornerstones of functional analysis, ensures that norm attaining functionals are dense in the dual of any Banach space, real or complex. The possibility of generalization of this profound result to the vector valued case has been studied by several mathematicians \cite{A,Da,Db}. The primary purpose of the present paper is to further explore the geometry of the space of bounded linear operators between Banach spaces with the same motivation. We illustrate that our study has natural connections with the study of smooth operators between Banach spaces. Before proceeding any further, let us fix the notations and the terminologies to be used throughout the paper.\\
 Let $ \mathbb{X}, $ $ \mathbb{Y} $ be Banach spaces. We work with only real Banach spaces of dimension greater than $ 1. $ Let $ B_{\mathbb{X}} = \{x \in \mathbb{X} : \|x\| \leq 1\} $ and $ S_{\mathbb{X}} = \{x \in \mathbb{X} : \|x\|=1\} $ be the unit ball and the unit sphere of $ \mathbb{X} $ respectively. Given any $ x \in \mathbb{X} $ and any $ r > 0, $ let $ B(x,r) $ denote the open ball with centre at $ x $ and radius $ r. $ Let $ \mathbb{L}(\mathbb{X},\mathbb{Y})~ (\mathbb{K}(\mathbb{X},\mathbb{Y})) $ denote the Banach space of all bounded (compact) linear operators from $ \mathbb{X} $ to $ \mathbb{Y}, $ endowed with the usual operator norm. Let $ \mathbb{X}^{*} $ denote the dual space of $ \mathbb{X}. $ Given $ T \in \mathbb{L}(\mathbb{X},\mathbb{Y}), $ let $ M_T $ denote the norm attainment set of $ T, $ i.e., $ M_T = \{x \in S_{\mathbb{X}} : \|Tx\|=\|T\|\}. $ The structure of $ M_T $ plays an important role in the geometry of operator spaces \cite{S,PSG}. In particular, the smoothness of a compact linear operator on a reflexive smooth Banach space is completely determined by the corresponding norm attainment set \cite{PSG}. An element $ \theta \neq x \in \mathbb{X} $ is said to be a smooth point if there exists a unique linear functional $ f \in \mathbb{X}^{*} $ such that $ \| f \| = 1 $ and $ f(x) = \| x \|. $ We make use of the notion of Birkhoff-James orthogonality in Banach spaces, towards obtaining results in the spirit of Bishop-Phelps theorem (or, to justify the impossibility of the existence of such results), for bounded linear operators instead of bounded linear functionals. Given any two elements $ x,~y \in \mathbb{X}, $ we say that $ x $ is Birkhoff-James orthogonal to $ y, $ written as $ x \perp_{B} y, $ if $ \| x + \lambda y \| \geq \| x \| $ for all scalars $ \lambda. $ Let us observe that the concepts of smoothness and Birkhoff-James orthogonality are applicable to bounded linear operators by treating them as elements of the corresponding Banach space of all bounded linear operators (between the same pair of Banach spaces). We say that $ \mathbb{X} $ is strictly convex if every point of $ S_{\mathbb{X}} $ is an extreme point of $ B_{\mathbb{X}}. $ Furthermore, $ \mathbb{X} $ is said to be locally uniformly rotund (LUR) if for all $ x,~x_n \in S_{\mathbb{X}} $ satisfying $ \lim_{n \to \infty} \| x_n+x \| = 2, $ we have that $ \lim_{n \to \infty} \| x_n - x \| = 0. $ LUR property is evidently stronger than strict convexity, only in the infinite-dimensional case. We also recall that $ \mathbb{X} $ is said to be Kadets-Klee space if whenever $ \{x_n\} $ is a sequence in $ \mathbb{X} $ and $ x \in \mathbb{X} $ is such that $ x_n \stackrel{w}{\rightharpoonup} x $ and $ \lim_{n \to \infty} \| x_n \| = \| x \|, $ then $ \lim_{n \to \infty} \| x_n - x \| = 0. $

While studying the numerical range of a bounded linear operator, Bollob$\acute{a}$s presented a quantitative version of the Bishop-Phelps theorem by proving the following result \cite{Ba}:\\  

\textit{Let $ \epsilon > 0 $ be arbitrary. If $ x \in B_{\mathbb{X}} $ and $ x^{*} \in S_{\mathbb{X}^{*}} $ are such that $ |1 - x^{*}(x)| < \frac{\epsilon ^{2}} {4}, $ then there are elements $ y \in S_{\mathbb{X}} $ and $ y^{*} \in S_{\mathbb{X}^{*}}, $ such that $ y^{*}(y) = 1,~ \| y-x \| < \epsilon $ and $ \| y^{*}-x^{*} \| < \epsilon. $}\\

Motivated by this novel idea, Acosta et al. introduced the following definition in their seminal paper \cite{A}:

\begin{definition}
Let $ \mathbb{X} $ and $ \mathbb{Y} $ be real or complex Banach spaces. We say that the couple $ (\mathbb{X},\mathbb{Y}) $ satisfies the Bishop-Phelps-Bollob$\acute{a}$s property (BPBp in abbreviated form) for operators if given $ \epsilon > 0, $ there are $ \eta(\epsilon) > 0 $ and $ \beta(\epsilon) > 0 $ with  $ lim_{t \rightarrow 0} \beta(t) = 0 $ such that for all $ T \in S_{\mathbb{L}(\mathbb{X},\mathbb{Y})}, $ if $ x_{0} \in S_{\mathbb{X}} $ is such that $ \| Tx_{0} \| > 1 - \eta(\epsilon), $ then there exist a point $ u_{0} \in S_{\mathbb{X}} $ and an operator $ S \in S_{\mathbb{L}(\mathbb{X},\mathbb{Y})} $ that satisfy the following conditions:\\
$ \| Su_{0} \| =1,~ \| u_0 -x_0 \| < \beta(\epsilon) $ and $ \| S-T \| < \epsilon. $
\end{definition}

They also proved in the same paper that for finite-dimensional Banach spaces, a uniform version of the Bishop-Phelps-Bollob$\acute{a}$s theorem holds for operators. In view of their result, it seems natural to introduce the following definition in the study of Bishop-Phelps-Bollob$\acute{a}$s type theorems for bounded linear operators between Banach spaces:

\begin{definition}
Let $ \mathbb{X} $ and $ \mathbb{Y} $ be Banach spaces and $ T \in S_{\mathbb{L}(\mathbb{X},\mathbb{Y})}. $ Let $ \epsilon > 0 $ be fixed. We say that $ A \in S_{\mathbb{L}(\mathbb{X},\mathbb{Y})} $ is a uniform $ \epsilon- $BPB approximation of $ T $ if there exists $ \delta(\epsilon) > 0 $ such that if $ x_{0} \in S_{\mathbb{X}} $ is such that $ \| Tx_{0} \| > 1 - \delta(\epsilon), $ then there exist a point $ u_{0} \in S_{\mathbb{X}} $ satisfying the following conditions:\\
$ \| Au_{0} \| =1,~ \| u_0 -x_0 \| < \epsilon $ and $ \| A-T \| < \epsilon. $
\end{definition}

Following this definition, Proposition $ 2.4 $ of \cite{A} can be reformulated in the following way:\\

\textit{Let $ \mathbb{X} $ and $ \mathbb{Y} $ be finite-dimensional Banach spaces and $ T \in S_{\mathbb{L}(\mathbb{X},\mathbb{Y})}. $ Then for every $ \epsilon > 0, $ $ T $ has a uniform $ \epsilon- $BPB approximation.}\\ 

Several authors have further studied Bishop-Phelps-Bollob$\acute{a}$s type theorems for bounded linear operators, with additional restrictions. Dantas introduced the following two definitions in \cite{Da} and obtained some interesting Bishop-Phelps-Bollob$\acute{a}$s type theorems (or counterexamples) for bounded (compact) linear operators in various special cases. 

\begin{definition}
Let $ \mathbb{X} $ and $ \mathbb{Y} $ be Banach spaces. We say that the pair $ (\mathbb{X},\mathbb{Y}) $ has $ \emph{property 1} $ (also called strong BPB property, or, sBPBp in abbreviated form), if given $ \epsilon > 0 $ and $ T \in S_{\mathbb{L}(\mathbb{X},\mathbb{Y})}, $ there exists $ \eta(\epsilon,T) > 0 $ such that whenever $ x_0 \in S_{\mathbb{X}} $ satisfies $ \| Tx_0 \| > 1-\eta(\epsilon,T), $ there exists $ x_1 \in S_{\mathbb{X}} $ such that $ \| Tx_1 \| = 1 $ and $ \| x_1 - x_0 \| < \epsilon. $ If this property is satisfied for every norm one compact operator, then we say that the pair $ (\mathbb{X},\mathbb{Y}) $ has sBPBp for compact operators.
\end{definition}

\begin{definition}
Let $ \mathbb{X} $ and $ \mathbb{Y} $ be Banach spaces. We say that the pair $ (\mathbb{X},\mathbb{Y}) $ has $ \emph{property 2} $ (also called uniform strong BPB property, or, uniform sBPBp in abbreviated form), if given $ \epsilon > 0, $ there exists $ \eta(\epsilon) > 0 $ such that whenever $ T \in S_{\mathbb{L}(\mathbb{X},\mathbb{Y})} $ and $ x_0 \in S_{\mathbb{X}} $ are such that $ \| Tx_0 \| > 1-\eta(\epsilon), $ there exists $ x_1 \in S_{\mathbb{X}} $ such that $ \| Tx_1 \| = 1 $ and $ \| x_1 - x_0 \| < \epsilon. $ 
\end{definition}

In \cite{Da}, Dantas proved that if $ \mathbb{X} $ is a reflexive Banach space which is locally uniformly rotund (LUR) then the pair $ (\mathbb{X},\mathbb{Y}) $ has sBPBp for compact operators. On the other hand, Dantas posed the following question in \cite{D}:\\

\textit{Is it possible to give a characterization for the pair $ (\mathbb{X},\mathbb{Y}) $ to have sBPBp?}\\ 

As regards to uniform sBPBp, very recently Dantas et al. proved in \cite{Db} that the pair $ (\mathbb{X},\mathbb{Y}) $ may have uniform sBPBp only if one of the spaces $ \mathbb{X} $ and $ \mathbb{Y} $ is one-dimensional. In view of this powerful result, it seems natural to introduce the following definition:

\begin{definition}
Let $ \mathbb{X} $ and $ \mathbb{Y} $ be Banach spaces. Let $ \mathcal{F} $ be a family of norm one linear operators in $ \mathbb{L}(\mathbb{X},\mathbb{Y}). $ We say that the pair $ (\mathbb{X},\mathbb{Y}) $ has uniform sBPBp with respect to $ \mathcal{F} $ if given $ \epsilon > 0, $ there exists $ \eta(\epsilon) > 0 $ such that whenever $ T \in \mathcal{F} $ and $ x_0 \in S_{\mathbb{X}} $ are such that $ \| Tx_0 \| > 1-\eta(\epsilon), $ there exists $ x_1 \in S_{\mathbb{X}} $ such that $ \| Tx_1 \| = 1 $ and $ \| x_1 - x_0 \| < \epsilon. $ 
\end{definition} 

In this paper, we study BPBp and its variants, from the point of view of operator norm attainment. As it turns out, a natural generalization of the norm attainment set of a bounded linear operator seems necessary to proceed in this direction. This motivates us to introduce the concept of approximate norm attainment set of a bounded linear operator, in the following natural way:

\begin{definition}
Let $ \mathbb{X} $ and $ \mathbb{Y} $ be Banach spaces and $ T \in \mathbb{L}(\mathbb{X},\mathbb{Y}) $ be nonzero. Let $ 0 < \delta < \| T \|. $ The $ \delta $-approximate norm attainment set of $ T, $ $ M_{T}(\delta) $ is defined as $ M_{T}(\delta) = \{x \in S_{\mathbb{X}} : \| Tx \| > \| T \| - \delta \}. $
\end{definition} 

We first obtain a characterization of the smoothness of a bounded linear operator in terms of the approximate norm attainment set of the operator, provided the domain space is reflexive and Kadets-Klee. We observe that it is also possible to answer the above mentioned question raised by Dantas in \cite{D}, by using the notion of approximate norm attainment set. We next focus on the uniform $ \epsilon-$BPB approximation of a bounded linear operator and obtain several interesting results, involving the norm attainment set of the operator. First, we extend Proposition $ 2.4 $ of \cite{A} to compact operators on a reflexive Kadets-Klee Banach space. As a consequence of our study, we obtain a complete characterization of smooth operators in the finite-dimensional case, in terms of the existence of nontrivial uniform $ \epsilon- $BPB approximations which are also smooth. We also prove a result in the opposite direction by establishing that given any isometry in $ \mathbb{L}(l_{p}^{2}, l_{p}^{2}), $ it is the only uniform $ \epsilon- $BPB approximation of itself, for sufficiently small $ \epsilon > 0. $ We further study uniform sBPBp for a pair of Banach spaces, with respect to a fixed family of norm one linear operators between them. We prove that if $ \mathbb{X} $ is a strictly convex and smooth Banach space then counterexamples to uniform sBPBp can already be found in the class of all norm one smooth operators. In other words, we prove that if $ \mathbb{X} $ is a strictly convex and smooth Banach space then the pair $ (\mathbb{X},\mathbb{X}) $ does not have uniform sBPBp with respect to the family of norm one smooth operators in $ \mathbb{L}(\mathbb{X},\mathbb{X}). $ We would like to remark that this result complements the deep result obtained by Dantas et al. in \cite{Db}, regarding the impossibility of uniform sBPBp between any pair of Banach spaces, if both of them have dimension greater than $ 1. $ We obtain a complete characterization of uniform sBPBp for a pair of Banach spaces, with respect to a given family of norm one bounded linear operators between them. As the final result of this paper, we prove that if $ \mathbb{X} $ is a reflexive Kadets-Klee Banach space and $ \mathbb{Y} $ is any Banach space, then the pair $ (\mathbb{X},\mathbb{Y}) $ has sBPBp for compact operators. We would like to remark that Dantas proved a similar result in \cite{Da}, assuming that $ \mathbb{X} $ is reflexive and LUR, instead of the Kadets-Klee property. Since it is well-known that every LUR space is necessarily Kadets-Klee, our result evidently covers the analogous result proved by Dantas. We end this paper by a remark that the converse is not true and therefore, our result is a proper refinement of the corresponding result by Dantas in \cite{Da}.      
  
\section{ Main Results.}

Let us begin with some basic properties of the approximate norm attainment set of a bounded linear operator. 

\begin{prop}
Let $ \mathbb{X},~\mathbb{Y} $ be Banach spaces and $ T \in \mathbb{L}(\mathbb{X},\mathbb{Y}) $ be nonzero. Let $ 0 < \delta,\delta_1,\delta_2 < \| T \|. $ Then the following are true:\\
$ (i) $ $ M_{T}(\delta) $ is nonempty.\\
$ (ii) $ $ \delta_1 < \delta_2 \implies M_{T}(\delta_1) \subseteq M_{T}(\delta_2).  $ Moreover, if $ T $ is not a scalar multiple of an isometry, then there exists $ 0 < \delta_{1} < \delta_{2} $ such that $ M_{T}(\delta_1) \subsetneq M_{T}(\delta_2). $  \\
$ (iii) $ $ M_{T} = {\bigcap}_{0 < \delta < \| T \|} M_{T}(\delta). $\\
$ (iv) $ If $ \mathbb{X} $ is finite-dimensional then $ T $ is injective if and only if for some $ 0 < \delta < \| T \|, $ we have that, $ M_{T}(\delta) = S_{\mathbb{X}}. $ This is not necessarily true if $ \mathbb{X} $ is infinite-dimensional.
\end{prop}

\begin{proof}
Every statement in Proposition $ 2.1 $ is trivial, except perhaps the last one. We note that since $ \mathbb{X} $ is finite-dimensional, we have that, $ S_{\mathbb{X}} $ is compact. Therefore, every bounded linear operator in $ \mathbb{L}(\mathbb{X},\mathbb{Y}) $ attains its minimum norm (say, $ k_{T} $) on $ S_{\mathbb{X}}, $ i.e., there exists $ x_0 \in S_{\mathbb{X}} $ such that $ \| Tx_0 \| = k_{T} = inf~ \{\| Tz \| : z \in S_{\mathbb{X}}\}.$ It is easy to observe that $ T \in \mathbb{L}(\mathbb{X},\mathbb{Y}) $ is injective if and only if $ k_{T} > 0. $ Now, if we choose $ 0 < \delta < \|T\|  $ to be such that $ \|T\| - \delta < k_{T}, $ then it is easy to observe that $ M_{T}(\delta) = S_{\mathbb{X}}. $ On the other hand, if $ 0 < \delta < \|T\|  $ is such that $ M_{T}(\delta) = S_{\mathbb{X}} $ then it follows immediately that $ k_{T} > 0. $ This establishes the first part of $ (iv). $\\  

To see that the last part of $ (iv) $ also holds true, consider $ T : \ell_2 \longrightarrow \ell_2 $ defined by $ T(a_n) = (\frac{1}{n}a_n). $ It is easy to observe that $ T $ is linear and $ \| T \|=1. $ Moreover, $ T $ is injective. However, a quick glance at the action of $ T $ on the canonical basis of $ \ell_2 $ ensures that there does not exist any $ \delta \in (0,1) $ such that $ M_{T}(\delta) = S_{\mathbb{X}}.$
\end{proof}

We now obtain a complete characterization of the smooth points in $ \mathbb{K}(\mathbb{X},\mathbb{Y}), $ where $ \mathbb{X} $ is a reflexive Kadets-Klee Banach space and $ \mathbb{Y} $ is a smooth Banach space, in terms of the approximate norm attainment set. We will use the following result, that follows from Theorem $ 4.1 $ and Theorem $ 4.2 $ of \cite{PSG}: \\

\begin{theorem}
Let $ \mathbb{X} $ be a reflexive Banach space and $ \mathbb{Y} $ be a smooth Banach space. Then $ T \in \mathbb{K}(\mathbb{X},\mathbb{Y}) $ is a smooth point if and only if  $ M_{T} = \{ \pm x_{0} \}, $ for some $ x_{0} \in S_{\mathbb{X}}. $ Moreover, for the ``only if" part, smoothness of $ \mathbb{Y} $ is not required.
\end{theorem}

An easy application of the above theorem yields the following result:

\begin{theorem}
Let $ \mathbb{X} $ be a reflexive Kadets-Klee Banach space and $ \mathbb{Y} $ be a smooth Banach space. Then $ T \in \mathbb{K}(\mathbb{X},\mathbb{Y}) $ is a smooth point if and only if there exists $ x_{0} \in S_{\mathbb{X}} $ such that given any $ \epsilon > 0, $ there exists $ \delta = \delta(\epsilon) > 0, $ satisfying the following:
\[ \textit{for any}~  0 < \delta \leq \delta(\epsilon),~  M_{T}(\delta) \subseteq B(x_0,\epsilon) \cup B(-x_0,\epsilon).  \]
Moreover, in the ``if" part of theorem, we do not require $ \mathbb{X} $ to be Kadets-Klee.
\end{theorem}

\begin{proof}

Let us first prove the ``if" part of the theorem. We observe that since $ \mathbb{X} $ is reflexive and $ T $ is compact, $ M_T $ is nonempty. Let $ \pm x_0 \in M_T. $ We claim that $ M_T = \{ \pm x_0 \}. $ If possible, suppose that there exists $ w_0 \in M_T $ such that $ w_0 \neq \pm x_0. $ Let us choose $ \epsilon_{0} > 0 $ such that $ \epsilon_0 < \frac{1}{2} min\{ \| x_0 - w_0 \|,~\| x_0 + w_0 \| \}. $ Now, we argue that it is impossible to find some $ \delta(\epsilon_0) > 0 $ and some $ z_0 \in S_{\mathbb{X}} $ such that for any $ 0 < \delta \leq \delta(\epsilon_0), $ $ M_{T}(\delta) \subseteq B(z_0,\epsilon_0) \cup B(-z_0,\epsilon_0). $ Indeed, for this to hold, we must have, either of the following is true: \\

\noindent $ (i) $ $ x_0 $ and $ w_0 $ belongs to a ball centered at $ z_0 $ (or $ - z_0 $) and radius $ \epsilon_0. $\\
$ (ii) $ $ -x_0 $ and $ w_0 $ belongs to a ball centered at $ z_0 $ (or $ - z_0 $) and radius $ \epsilon_0. $\\

In the first case, $ \| x_0 - w_0 \| < 2\epsilon_0, $ whereas, in the second case, $ \| x_0 + w_0 \| < 2\epsilon_0. $ In both the cases, we arrive at a contradiction to our initial choice of $ \epsilon_0. $ This completes the proof of our claim. Let us now observe that since $ \mathbb{Y} $ is smooth, the ``if" part of the theorem follows directly from Theorem $ 2.2. $ We would like to further note that the Kadets-Klee property of $ \mathbb{X} $ is not required to complete the proof of this part of the theorem.\\

Let us now prove the ``only if" part. Since $ \mathbb{X} $ is reflexive and $ T $ is smooth, $ M_{T}=\{ \pm x_0 \}, $ for some $ x_0 \in S_{\mathbb{X}}. $ Clearly, $ T $ is nonzero. Let $ \epsilon > 0 $ be given arbitrarily. We claim that there exists $ \delta_{0} > 0 $ such that for any $ y \in S_{\mathbb{X}} \setminus (B(x_0,\epsilon) \cup B(-x_0,\epsilon)), $ we have, $ \| Ty \| < \| T \| - \delta_0.  $ If this does not hold true then there exists a sequence $ \{y_n\} $ in   $ S_{\mathbb{X}} \setminus (B(x_0,\epsilon) \cup B(-x_0,\epsilon)) $ such that $ \| Ty_n \| \rightarrow \| T \| $ as $ n \rightarrow \infty. $ Since $ \mathbb{X} $ is reflexive, $ B_{\mathbb{X}} $ is weakly compact. Therefore, without loss of generality, we may and do assume that  $ y_n \stackrel{w}{\rightharpoonup} y_0 \in B_{\mathbb{X}} $ (say). Since $ T $ is compact and norm is a continuous function, we can easily deduce that $ \| Ty_n \| \rightarrow \| Ty_0 \|. $ Since we have already assumed that $ \| Ty_n \| \rightarrow \| T \|, $ it follows that $ \| Ty_0 \| = \| T \|. $ At this point of the proof, we recall that $ y_0 \in B_{\mathbb{X}}. $ Therefore, $ \| Ty_0 \| = \| T \| $ implies that $ y_0 \in S_{\mathbb{X}}. $ Since $ M_T = \{ \pm x_0 \}, $ we must have, $ y_0 = \pm x_0. $ On the other hand, we note that $ y_n \stackrel{w}{\rightharpoonup} y_0 $ and $ \| y_n \| = \| y_0 \| = 1, $ for every natural number $ n. $ Since $ \mathbb{X} $ is Kadets-Klee, this implies that $ y_n \rightarrow y_0 $ in norm. However, this clearly leads to a contradiction, since each $ y_n, $ being an element of  $ S_{\mathbb{X}} \setminus (B(x_0,\epsilon) \cup B(-x_0,\epsilon)), $ is at a distance of at least $ \epsilon $ from either of $ \pm x_0. $ This contradiction completes the proof of our claim. Now, let us choose $ \delta(\epsilon) = \delta_0. $ It is clear from our construction that if $ z \in S_{\mathbb{X}} $ is such that $ \| Tz \| > \| T \| - \delta(\epsilon), $ then we must have, $ z \in B(x_0,\epsilon) \cup B(-x_0,\epsilon). $ The proof of the ``only if" part now follows from the Statement $ (ii) $ of Proposition $ 2.1. $ This establishes the theorem. 
\end{proof}

\begin{remark}
We would like to note that in the ``only if" part of the above theorem, smoothness of $ \mathbb{Y} $ is not required. However, in order to obtain a complete characterization of smooth operators, we do require the additional assumption of smoothness. It is  also worth mentioning that for the above theorem to be true, it suffices to assume that $ Tx_0 $ is a smooth point in $ \mathbb{Y}, $ instead of the global smoothness of $ \mathbb{Y}. $ Let us observe that since $ T $ cannot be the zero operator, $ Tx_0 $ is nonzero, and therefore, smoothness of $ \mathbb{Y} $ at the point $ Tx_0 $ makes sense.
\end{remark}

It is quite straightforward to observe that using the concept of approximate norm attainment set of a bounded linear operator, it possible to rephrase the strong BPB property for a pair of Banach spaces. This answers the question regarding a complete characterization of sBPBp, raised by Dantas in \cite{D}.

\begin{theorem}
Let $ \mathbb{X},~\mathbb{Y} $ be Banach spaces. Then the pair $ (\mathbb{X},\mathbb{Y}) $ has sBPBp if and only if given any $ \epsilon > 0, $ and any $ T \in \mathbb{L}(\mathbb{X},\mathbb{Y}), $ there exists $ \delta = \delta(\epsilon,T) > 0 $ such that\\
\[ M_{T}(\delta) \subseteq \bigcup_{x \in M_{T}} (B(x,\epsilon) \cap S_{\mathbb{X}}). \] 
\end{theorem}

\begin{proof}
The proof follows quite trivially from the very definitions of sBPBp and $ M_{T}(\delta). $ Towards proving either the ``if" part or the ``only if" part, it suffices to choose $ \eta(\epsilon,T) = \delta(\epsilon,T). $
\end{proof}

Let us now focus on the uniform $ \epsilon- $BPB approximation of a bounded linear operator between Banach spaces. As mentioned in the introduction, our starting point in this aspect is Proposition $ 2.4 $ of \cite{A}. We would like to extend this result for compact operators defined on a reflexive Banach space with Kadets-Klee property. To do this, we prove that every compact operator defined on a reflexive Kadets-Klee Banach space is a uniform $ \epsilon- $BPB approximation of itself for any given $ \epsilon > 0. $

\begin{theorem}
Let $ \mathbb{X} $ be a reflexive Kadets-Klee Banach space and $ \mathbb{Y} $ be any Banach spaces. Let $ T \in \mathbb{K}(\mathbb{X},\mathbb{Y}) $ be of norm one and $ \epsilon > 0 $ be fixed. Then $ T $ is a uniform $ \epsilon- $BPB approximation of itself. In particular, if $ \mathbb{X} $ and $ \mathbb{Y} $ are finite-dimensional Banach spaces, then every $ T \in \mathbb{L}(\mathbb{X},\mathbb{Y}) $ has a uniform $ \epsilon- $BPB approximation, for every $ \epsilon > 0. $ 
\end{theorem}

\begin{proof}
Since $ \mathbb{X} $ is reflexive and $ T $ is compact, we have that, $ M_{T} \neq \emptyset. $ Let $ \epsilon > 0 $ be arbitrary. Let us consider the following open set: $ \textit{O} = \bigcup_{x \in M_{T}} B(x,\epsilon).  $ Following the same line of arguments, as given in the proof of the ``only if" part of Theorem $ 2.3, $ we can show that there exists $ \delta_{0} > 0 $ such that for any $ y \in S_{\mathbb{X}} \setminus \textit{O}, $ we have, $ \| Ty \| < \| T \| - \delta_0. $ We would like to note that in this part of the proof, we require that $ \mathbb{X} $ is Kadets-Klee. Let us choose $ \delta(\epsilon) = \delta_0. $ Therefore, if $ x_0 \in S_{\mathbb{X}} $ is such that $ \| Tx_0 \| > \| T \| - \delta(\epsilon), $ then $ x_0 \in \bigcup_{x \in M_{T}} B(x,\epsilon). $ In other words, there exists $ u_0 \in S_{\mathbb{X}}, $ such that $ \| Tu_0 \| = \| T \| $ and $ \| u_0 - x_0 \| < \epsilon. $ Since $ \| T - T \| = 0 < \epsilon, $ it follows that $ T $ is a uniform $ \epsilon- $BPB approximation of itself. This completes the proof of the first part of the theorem. Since every finite-dimensional Banach space is reflexive, the second part of the theorem follows immediately. This establishes the theorem.
\end{proof}

In view of the above theorem, it seems natural to investigate when does an operator have a $ \textit{nontrivial} $ uniform $ \epsilon- $BPB approximation, for every $ \epsilon > 0. $ It turns out that this question has a natural connection with the study of smooth operators between Banach spaces. First, we obtain a necessary condition for smoothness of a compact operator between a reflexive Kadets-Klee Banach space and an arbitrary Banach space.

\begin{theorem}
Let $ \mathbb{X} $ be a reflexive Kadets-Klee Banach space and $ \mathbb{Y} $ be any Banach space. Let $ T \in \mathbb{K}(\mathbb{X},\mathbb{Y}) $ be a norm one smooth operator in $ \mathbb{K}(\mathbb{X},\mathbb{Y}). $ Then $ T $ admits a nontrivial uniform $ \epsilon- $BPB approximation in $ \mathbb{L}(\mathbb{X},\mathbb{Y}) $, for each $ \epsilon > 0. $ 
\end{theorem}

\begin{proof}
Since $ \mathbb{X} $ is reflexive and $ T \in \mathbb{K}(\mathbb{X},\mathbb{Y}) $ is smooth, we have, $ M_{T} = \{ \pm x_0 \}, $ for some $ x_0 \in S_{\mathbb{X}}. $ Let $ \epsilon > 0 $ be fixed. It follows from the Hahn-Banach theorem that there exists a hyperplane $ H $ of codimension $ 1 $ in $ \mathbb{X} $ such that $ x_0 \perp_{B} H, $ i.e., $ x_0 \perp_{B} y $ for all $ y \in H. $ Furthermore, it is clear that every element of $ \mathbb{X} $ can be uniquely written as $ \alpha x_0 + h, $ where $ \alpha \in \mathbb{R} $ and $ h \in H. $ For each natural number $ n, $ we define a map $ A_n:\mathbb{X} \rightarrow \mathbb{Y} $ in the following way: \\
\[ A_n(\alpha x_0 + h) = \alpha Tx_0 + (1 - \frac{1}{n}) Th. \]

It is clear that each $ A_n $ is well-defined and linear. We claim that for each $ n \in \mathbb{N}, $ $ \| A_n \| = 1, $ and therefore, in particular, $ A_n \in \mathbb{L}(\mathbb{X},\mathbb{Y}), $ for each  $ n. $ To prove our claim, we proceed in the following way:\\
Let $ z = \alpha x_0 + h \in S_{\mathbb{X}}. $ We note that, for any nonzero $ \alpha, $ $ 1 = \| z \| = \| \alpha x_0 + h \| = | \alpha | \|x_0 + \frac{1}{ \alpha } h\| \geq | \alpha |. $ Therefore, we have, for any $ z \in S_{\mathbb{X}}, $

\begin{eqnarray*}
        \| A_n z \|&=&\| \alpha Tx_0 + (1 - \frac{1}{n}) Th \|\\
      &=&\|(1-\frac{1}{n})(\alpha Tx_0+Th)+\frac{1}{n}\alpha Tx_0\|\\
 & \leq & (1-\frac{1}{n}) \| Tz \| + \frac{1}{n} | \alpha | \|Tx_0\|\\
 & \leq & (1-\frac{1}{n}) + \frac{1}{n}\\
 & = & 1 
\end{eqnarray*}

Since $ \| A_n x_0 \| = \| Tx_0 \| = \| T \| = 1, $ it follows that $ \| A_n \| = 1, $ for each $ n. $ This completes the proof of our claim. We observe that  for each $ n, $ $ A_n \neq T $ and $ x_0 \in M_{A_{n}}. $ Furthermore, for $ z=\alpha x_0 + h \in S_{\mathbb{X}}, $ we also observe that $ 1 = \| \alpha x_0 + h \| \geq \| h \| - | \alpha | \geq \| h \| - 1. $ In other words, for any $ z=\alpha x_0 + h \in S_{\mathbb{X}}, $ we have, $ \| h \| \leq 2. $ Therefore, $ \| (T - A_n)z \| = \frac{1}{n} \| Th \| \leq \frac{2}{n} \| T \| \rightarrow 0 $ as $ n \rightarrow \infty. $ In particular, this allows us to conclude that $ \| T - A_n \| < \epsilon, $ whenever $ n $ is sufficiently large.\\
Since $ M_T = \{ \pm x_0 \},~\mathbb{X} $ is reflexive, Kadets-Klee and $ T $ is compact, it follows from the arguments given in the proof of the ``only if" part of Theorem $ 2.3 $ that there exists $ \delta = \delta(\epsilon) > 0 $ such that $ M_{T}(\delta) \subseteq (B(x_0,\epsilon) \cup B(-x_0,\epsilon)). $ In other words, if $ y \in S_{\mathbb{X}} $ is such that $ \| Ty \| > \| T \| - \delta, $ then for sufficiently large $ n, $ we have the following:\\

$ (i)~ \| A_n x_0 \| = \| A_n \| = 1,~(ii)~ \| y - x_0 \| < \epsilon,~(iii)~ \| T - A_n \| < \epsilon $ and $ (iv)~ A_n \neq T. $\\ 

Equivalently, we have that, each $ A_n $ is a nontrivial uniform $ \epsilon- $BPB approximation of $ T, $ for sufficiently large $ n.$ Since $ \epsilon > 0 $ was chosen arbitrarily, this establishes the theorem.
\end{proof}

Conversely, in the following proposition, we would like to prove a sufficient condition for smoothness in the space of compact operators, when the domain space is reflexive and the range space is smooth. 

\begin{prop}
Let $ \mathbb{X} $ be a reflexive Banach space and $ \mathbb{Y} $ be a  smooth Banach space. Let $ T \in \mathbb{K}(\mathbb{X},\mathbb{Y}) $ be of norm one. Suppose that for every $ \epsilon > 0, $ $ T $ admits a nontrivial uniform $ \epsilon- $BPB approximation in $ \mathbb{K}(\mathbb{X},\mathbb{Y}) $ which is smooth in $ \mathbb{K}(\mathbb{X},\mathbb{Y}). $ Then $ T $ itself is smooth in $ \mathbb{K}(\mathbb{X},\mathbb{Y}). $
\end{prop}

\begin{proof}
Clearly, $ \pm x_0 \in M_{T}, $ for some $ x_0 \in S_{\mathbb{X}}. $ We claim that $ M_{T} = \{ \pm x_0 \}, $ or, equivalently, $ T $ is smooth. If possible, suppose that there exists $ w_0 \in M_T $ such that $ w_{0} \neq \pm x_0. $ Let $ A_{\epsilon} \in \mathbb{L}(\mathbb{X},\mathbb{Y})$ be a uniform $ \epsilon- $BPB approximation of $ T $ such that each $ A_{\epsilon} $ is smooth. Let us choose $ \epsilon_0 > 0 $ such that $ \epsilon_0 < \frac{1}{2}min\{ \| x_0 - w_0 \|,~\| x_0+w_0 \| \}. $ Let $ \delta = \delta(\epsilon_0) > 0 $ be the constant of uniform $ \epsilon- $BPB approximation, corresponding to the value $ \epsilon = \epsilon_{0}. $  Now, $ \| Tx_0 \| = \| Tw_0 \| = 1 > 1 - \delta. $ Since $ A_{\epsilon_{0}} $ is a uniform $ \epsilon_0- $BPB approximation of $ T, $ the following must be true:\\

$ A_{\epsilon_{0}} $ attains norm in each of the open balls $ B(x_0,\epsilon_{0}),~B(-x_0,\epsilon_{0})  $ and $ B(w_0,\epsilon_{0}). $\\

By virtue of our choice of $ \epsilon_0, $ it is easy to see that these three balls are mutually disjoint. Therefore, it follows that $ A_{\epsilon_{0}} $ must attain norm at more than one pair of points. However, this contradicts our assumption that $ A_{\epsilon_{0}} $ is smooth. This contradiction completes the proof of our claim and establishes the theorem. 
\end{proof}

Combining Theorem $ 2.6 $ and Proposition $ 2.7, $ it is possible to completely characterize smooth points in the operator space in the finite-dimensional case, if we further assume that the domain space is strictly convex. We accomplish this goal in the next theorem:

\begin{theorem}
Let $ \mathbb{X} $ be a finite-dimensional strictly convex Banach space and $ \mathbb{Y} $ be a finite-dimensional smooth Banach space. Let $ T \in \mathbb{L}(\mathbb{X},\mathbb{Y}). $ Then $ T $ is smooth if and only if for every $ \epsilon > 0, $ $ T $ admits a nontrivial uniform $ \epsilon- $BPB approximation which is also smooth.
\end{theorem} 

\begin{proof}
Clearly, the ``if" part is already proven. Let us prove the ``only if" part of the theorem. We note that, in this case, $ \mathbb{K}(\mathbb{X},\mathbb{Y}) = \mathbb{L}(\mathbb{X},\mathbb{Y}), $ and, furthermore, $ \mathbb{X} $ is reflexive and Kadets-Klee. We follow the notations from the proof of Theorem $ 2.6. $ It is obvious that the proof will be completed, if we can show that each $ A_n $ is smooth, or, equivalently, each $ A_n $ attains norm at only one pair of points. We claim that $ M_{A_{n}} = \{ \pm x_0 \} $ for each $ n. $    Let $ z=\alpha x_0 + h \in S_{\mathbb{X}}. $ This is where we would like to apply the strict convexity of $ \mathbb{X}. $ Suppose $ \alpha \neq 0. $ We have, $ 1=\|z\|=|\alpha| \|x_0+\frac{1}{\alpha}h\| \geq |  \alpha |, $ since $ x_0 \perp_{B} h. $ Let us note that we must have, $ x_0,~h $ are linearly independent, provided $ h \neq 0. $ As $ \mathbb{X} $ is strictly convex, we have, $ \|x_0+\frac{1}{\alpha}h\| > \| x_0 \| = 1, $ whenever $ h \neq 0. $ In other words, whenever $ h \neq 0, $ we have that,  $ | \alpha | < 1. $ On the other hand, clearly, $ h=0 $ implies that $ \alpha = \pm 1, $ i.e., $ z=\pm x_0. $ Therefore, in effect, we have proved the following:\\ 

If $z=\alpha x_0 + h \in S_{\mathbb{X}}$ then $ | \alpha | \leq 1. $ Moreover, $ | \alpha |=1 $ if and only if $ z=\pm x_0. $\\

\noindent Now, for any $ z \in S_{\mathbb{X}} \setminus \{ \pm x_0 \}, $ we have,

\begin{eqnarray*}
        \| A_n z \|&=&\| \alpha Tx_0 + (1 - \frac{1}{n}) Th \|\\
      &=&\|(1-\frac{1}{n})(\alpha Tx_0+Th)+\frac{1}{n}\alpha Tx_0\|\\
 & \leq & (1-\frac{1}{n}) \| Tz \| + \frac{1}{n} | \alpha | \|Tx_0\|\\
 & < & (1-\frac{1}{n}) + \frac{1}{n}\\
 & = & 1 
\end{eqnarray*}

This proves that $ M_{A_{n}} = \{ \pm x_0 \}, $ as claimed by us. This completes the proof of the theorem.
\end{proof}

In light of Theorem $ 2.8, $ the following complementary question arises naturally:\\

\textit{Let $ \mathbb{X}, \mathbb{Y} $ be finite-dimensional strictly convex and smooth Banach spaces. Obtain a necessary and sufficient condition for a norm one element $ T $ of $ \mathbb{L}(\mathbb{X},\mathbb{Y}) $  to be such that $ T $ admits no nontrivial uniform $ \epsilon- $BPB approximation for sufficiently small $ \epsilon > 0. $}\\

While we are unable to answer this general question, we prove that the class of all norm one linear operators satisfying the above mentioned property is nonempty in $ \mathbb{L}(\ell_{p}^{2},\ell_{p}^{2}), $ where $ 2 < p \in \mathbb{N}. $ Indeed, in the following theorem, we prove that any isometry in $ \mathbb{L}(\ell_{p}^{2},\ell_{p}^{2}) $ belongs to the desired category, when $ 2 < p \in \mathbb{N}. $

\begin{theorem}
Let $ \mathbb{X}=\ell_{p}^{2}, $ where $ 2 < p \in \mathbb{N}. $ Let $ T \in \mathbb{L}(\mathbb{X},\mathbb{X}) $ be an isometry. Then there exists $ \epsilon_0 > 0 $ such that for any $ 0 < \epsilon < \epsilon_0, $ $ T $ is the only uniform $ \epsilon- $BPB approximation of $ T. $ 
\end{theorem} 

\begin{proof}
We begin our proof with the remark that it is well-known that there are only finitely many isometries in $ \mathbb{L}(\mathbb{X},\mathbb{X}). $\\
Let $ \epsilon_1 =  min~\mbox{\{ $ \| V-S \| $ : $ V,~S $ are distinct isometries in $ \mathbb{L}(\mathbb{X},\mathbb{X}) $ \} } > 0. $\\
We would further like to remark that it follows from Theorem $ 2.8 $ of \cite{S} that any linear operator in $ \mathbb{L}(\mathbb{X},\mathbb{X}), $ which is not a scalar multiple of an isometry, attains norm at not more than $ 2(8p-5) $ number of points of $ S_{\mathbb{X}}. $\\
\\
Let us choose $ 0 < \epsilon_{0} < min~\{ \epsilon_1,\frac{1}{2(8p-5)} \}. $\\

We claim that for any $ 0 < \epsilon \leq \epsilon_0, $ $ T $ is the only uniform $ \epsilon- $BPB approximation of $ T. $ If $ S (\neq T) $ is any isometry in $ \mathbb{L}(\mathbb{X},\mathbb{X}), $ then $ S $ cannot be a uniform $ \epsilon- $BPB approximation of $ T, $ since $ \| S - T \| > \epsilon_0 \geq \epsilon . $ Now, let $ A \in \mathbb{L}(\mathbb{X},\mathbb{X}) $ be a norm one operator which is not an isometry. Then $ |M_{A}| \leq 2(8p-5), $ where $ | M_{A} | $ denotes the cardinality of $ A. $ Let $ A $ attains norm only at the points $ \pm x_1,\pm x_2,\ldots,\pm x_k \in S_{\mathbb{X}}, $ where $ k \leq (8p-5). $ Let $ \textit{O}=\bigcup_{i=1}^{k} B(\pm x_i,\epsilon). $ We observe that the diameter of $ \textit{O}=sup~ \{ \|x-y\| : x, y \in \textit{O} \} \leq 4k\epsilon \leq 4(8p-5)\epsilon_0 < 2. $ On the other hand, the diameter of $ S_{\mathbb{X}} = sup~\{ \|x-y\| : x,y \in S_{\mathbb{X}} \} = 2. $ This proves that $ \textit{O}\cap S_{\mathbb{X}} \subsetneq S_{\mathbb{X}}. $ In other words, $ S_{\mathbb{X}} \setminus \textit{O} $ is nonempty. Let us choose $ z_0 \in S_{\mathbb{X}} \setminus \textit{O}. $ Since $ T $ is an isometry, $ z_0 \in M_{T}. $ However, our choice of $ z_0 \in S_{\mathbb{X}} \setminus \textit{O} $  ensures that $ A $ does not attain norm in an $ \epsilon $ neighbourhood of the point $ z_0. $ In other words, $ A $ cannot be a uniform $ \epsilon- $BPB approximation of $ T. $ We note that $ A $ was chosen arbitrarily, with the only restrictions that $ \| A \|=1 $ and $ A $ is not an isometry. Therefore, we have effectively proved that when $ \epsilon \leq \epsilon_0, $ $ T $ cannot have a nontrivial uniform $ \epsilon- $BPB approximation. On the other hand, Theorem $ 2.5 $ ensures that $ T $ is a uniform $ \epsilon- $BPB approximation of $ T, $ for every $ \epsilon > 0. $ Combining all these facts together, we may and do conclude that our claim is true. This establishes the theorem.
\end{proof}

Our next objective is to show that in certain cases, counterexamples to uniform sBPBp can already be found by considering only the class of smooth operators of norm one.

\begin{theorem}
Let $ \mathbb{X} $ be a finite-dimensional strictly convex and smooth Banach space. Let $ \mathcal{F} $ be the class of all norm one smooth operators in $ \mathbb{L}(\mathbb{X},\mathbb{X}). $ Then the pair $ (\mathbb{X},\mathbb{X}) $ does not have the uniform sBPBp with respect to $ \mathcal{F}. $
\end{theorem} 

\begin{proof}
Let us fix any $ x_0 \in S_{\mathbb{X}}. $ Let $ H_0 $ be the hyperplane of codimension $ 1 $ such that $ x_0 \perp_{B} H_0, $ i.e., $ x_0 \perp_{B} y_0 $ for all $ y_{0} \in H_0. $ For each natural number $ n, $ we define a map $ A_n : \mathbb{X} \rightarrow \mathbb{X} $ in the following way:\\

$ A_n (\alpha x_0 + h_0) = \alpha x_0 + (1-\frac{1}{n}) h_0, $ for $ \alpha \in \mathbb{R} $ and $ h_0 \in H_0. $\\

Since $ \mathbb{X} $ is strictly convex, following arguments similar to those given in the proof of Theorem $ 2.8, $ it is easy to verify that for each $ n \in \mathbb{N}, $ $ \| A_n \| =1 $ and $ M_{A_{n}} = \{ \pm x_0 \}.$ Since $ \mathbb{X} $ is finite-dimensional and smooth, this proves that each $ A_n $ is a norm one smooth point in $ \mathbb{L}(\mathbb{X},\mathbb{X}), $ i.e., $ A_n \in \mathcal{F}. $ On the other hand, if $ y_0 \in H_0 \cap S_{\mathbb{X}} $ is any any vector then it is easy to see that $ \| A_n y_0 \| = \| (1-\frac{1}{n}) y_0 \| = 1-\frac{1}{n} \rightarrow 1 $ as $ n \rightarrow \infty. $ \\
We note that $ \|x_0 - y_0\|,~\|x_0 + y_0\| \geq 1, $ since $ x_0 \perp_{B} y_0. $ If possible, suppose that $ (\mathbb{X},\mathbb{X}) $ has uniform sBPBp with respect to $ \mathcal{F}. $ Let us choose $ \epsilon = \frac{1}{2}. $ Then there exists $ \eta=\eta(\epsilon)>0 $ such that whenever $ T \in \mathcal{F} $ and $ z_0 \in S_{\mathbb{X}} $ are such that $ \|Tz_0\|>1-\eta, $ there exists $ z_1 \in S_{\mathbb{X}} $ such that $ \|Tz_1\|=1 $ and $ \| z_1 - z_0 \| < \epsilon.  $  Since $ \|A_n y_0\| \rightarrow 1, $ there exists $ n_0 \in \mathbb{N} $ such that $ \| A_{n_{0}} y_0 \| > 1-\eta. $ However, $ A_{n_{0}} $ does not attain norm in an $ \epsilon $ neighbourhood of $ y_0, $ since $ M_{A_{n_{0}}} = \{ \pm x_0 \} $ and $ \| x_0 - y_0 \|,~\|x_0 + y_0\| \geq 1 > \epsilon. $ Since $ A_{n_{0}} \in \mathcal{F}, $ this contradicts our hypothesis that $ (\mathbb{X},\mathbb{X}) $ has uniform sBPBp with respect to $ \mathcal{F}. $ This contradiction completes the proof of the theorem.  
\end{proof}

In the next theorem, we would like to obtain a complete characterization of uniform sBPBp for a pair of Banach spaces, with respect to a given family of norm one bounded linear operators. To serve this purpose, it is convenient to introduce a new notation.  

\begin{theorem}
Let $ \mathbb{X}~\mathbb{Y} $ be Banach spaces. Let $ \mathcal{F} $ be a family of norm one bounded linear operators in $ \mathbb{L}(\mathbb{X},\mathbb{Y}). $ Then the pair $ (\mathbb{X},\mathbb{Y}) $ has uniform sBPBp with respect to $ \mathcal{F} $ if and only if for every $ \epsilon > 0, $ we have,  $ sup~\{ \| Tz \| : T \in \mathcal{F}, z \in D(T,\epsilon) \} < 1, $ where $ D(T,\epsilon) = S_{\mathbb{X}} \setminus (\bigcup_{x \in M_{T}} B(x,\epsilon)).$ 
\end{theorem}

\begin{proof}
Let us first prove the ``if" part. Let $ \epsilon > 0 $ be fixed. Let $ sup~\{ \| Tz \| : T \in \mathcal{F}, z \in D(T,\epsilon) \} = 1 - \delta.$ Let us choose $ \eta = \eta(\epsilon) = \delta. $ Let us suppose that $ T \in \mathcal{F} $ and $ x_0 \in S_{\mathbb{X}} $ are such that $ \| Tx_0 \| > 1 - \eta. $ It is clear from the choice of $ \eta $ that we must have, $ x_0 \notin D(T,\epsilon) = S_{\mathbb{X}} \setminus \bigcup_{x \in M_{T}} B(x,\epsilon). $ Since $ x_0 \in S_{\mathbb{X}}, $ it follows that $ x_{0} \in \bigcup_{x \in M_{T}} B(x,\epsilon). $ In other words, there exists $ x_1 \in M_{T} $ such that $ \| x_1 - x_0 \| < \epsilon. $ Since $ \epsilon > 0 $ is chosen arbitrarily, this proves that the pair $ (\mathbb{X},\mathbb{Y}) $ has uniform sBPBp and thereby establishes the theorem in one direction.\\
The ``only if" part of the theorem can be proved by applying similar reasoning. Suppose the pair $ (\mathbb{X},\mathbb{Y}) $ has uniform sBPBp with respect to $ \mathcal{F}. $ Given any $ \epsilon > 0, $ let $ \eta=\eta(\epsilon) $ be the constant associated with uniform sBPBp with respect to $ \mathcal{F}, $ corresponding to $ \epsilon. $ Then it is easy to see that  $ sup~\{ \| Tz \| : T \in \mathcal{F}, z \in D(T,\epsilon) \} \leq 1 - \eta < 1. $ This completes the proof of theorem in the reverse direction. 
\end{proof}

As the final result of this paper, we obtain a sufficient condition for a pair of Banach spaces to satisfy sBPBp for compact operators. It will be clear from the remark immediately after the theorem that our result is a proper refinement of an earlier analogous result by Dantas in \cite{Da}.

\begin{theorem}
Let $ \mathbb{X} $ be a reflexive Kadets-Klee Banach space and $ \mathbb{Y} $ be any Banach space. Then the pair $ (\mathbb{X},\mathbb{Y}) $ has sBPBp for compact operators.
\end{theorem}

\begin{proof}
We prove the result by the method of contradiction. If possible, suppose that the pair $ (\mathbb{X},\mathbb{Y}) $ does not have sBPBp for compact operators. Then there exists an $ \epsilon_0 > 0 $ and a norm one compact operator $  T_0 \in \mathbb{K}(\mathbb{X},\mathbb{Y}) $ such that no $ \eta = \eta(\epsilon_0,T_0) $ ``works", i.e., given any $ \eta > 0, $ the following does not hold:\\

$ \| T_0 x_0 \| > 1 - \eta $ for some $ x_0 \in S_{\mathbb{X}} $ implies that there exists $ x_1 \in S_{\mathbb{X}} $ such that $ \| T_0 x_1 \| = 1 $ and $ \| x_1 - x_0 \| < \epsilon. $ \\

In particular, for any natural number $ n, $ there exists $ x_n \in S_{\mathbb{X}} $ such that $ \| T_0 x_n \| > 1 - \frac{1}{n} $ and $ \| x_n - y \| \geq \epsilon_0 $ for any $ y \in M_{T_{0}}. $ Since $ \mathbb{X} $ is reflexive and $ \| x_n \| = 1 $ for each $ n, $ we must have, $ \{ x_n \} $ has a weakly convergent subsequence. Without loss of generality, let us assume that  $ x_n \stackrel{w}{\rightharpoonup} z_0 $ (say). Since $ T_0 $ is compact and norm is a continuous function, we have that, $ \| T_0 x_n \| \rightarrow \| T_0 z_0 \|. $ However, since $ 1 - \frac{1}{n} < \| T_0 x_n \| \leq 1 = \| T_0 \|, $ it follows that $ z_0 \in M_{T_{0}}. $ In particular, we have that, $ z_0 \in S_{\mathbb{X}}. $ Since $ \mathbb{X} $ is Kadets-Klee,  $ x_n \stackrel{w}{\rightharpoonup} z_0 $ and $ 1 = \| x_n \| = \| z_0 \| $ for each $ n, $ we must have, $ x_n \rightarrow z_0 \in M_{T_{0}}. $ However, this is clearly in contradiction with our assumption that $ \| x_n - y \| \geq \epsilon_0 $ for any $ y \in M_{T_{0}}. $ This contradiction completes the proof of the theorem.

\end{proof}

\begin{remark}
Dantas proved in \cite{Da} that if $ \mathbb{X} $ is reflexive and LUR then for any Banach space $ \mathbb{Y}, $ the pair $ (\mathbb{X},\mathbb{Y}) $ has sBPBp for compact operators. It is well-known that LUR implies Kadets-Klee property. Therefore, the analogous result proved by Dantas in \cite{Da} follows directly from our result. On the other hand, in \cite{PS}, Polak and Sims have given an example of a Banach space which is reflexive and Kadets-Klee but not LUR. This evidently illustrates that our result is an improvement of the corresponding result by Dantas in \cite{Da}.   
\end{remark}

\bibliographystyle{amsplain}

\end{document}